\documentclass[12pt,letterpaper,reqno]{amsart}

\usepackage{times}
\usepackage[T1]{fontenc}
\usepackage{mathrsfs}

\usepackage{latexsym}
\usepackage[dvips]{graphics}
\usepackage{epsfig}
\usepackage{amsmath,amsfonts,amsthm,amssymb,amscd}
\input amssym.def
\input amssym.tex

\makeindex

\newcommand{\ep}[1]{ \mathbf{e}{ \underline{#1} \choose p} }

\newcommand\be{\begin{equation}}
\newcommand\ee{\end{equation}}
\newcommand\bea{\begin{eqnarray}}
\newcommand\eea{\end{eqnarray}}
\newcommand\bi{\begin{itemize}}
\newcommand\ei{\end{itemize}}
\newcommand\ben{\begin{enumerate}}
\newcommand\een{\end{enumerate}}
\newcommand\bc{\begin{center}}
\newcommand\ec{\end{center}}
\newcommand\ba{\begin{array}}
\newcommand\ea{\end{array}}

\newcommand{\Z}{\ensuremath{\mathbb{Z}}}
\newcommand{\Q}{\mathbb{Q}}

\newcommand{\Qoft}{\mathbb{Q}(t)}  

\newcommand{\foh}{\frac{1}{2}}  

\newcommand{\js}[1]{ { \underline{#1} \choose p} }

\newtheorem{thm}{Theorem}[section]

\newtheorem{lem}[thm]{Lemma}

\newtheorem{rek}[thm]{Remark}

\newcommand{\ncr}[2]{{#1 \choose #2}}
\newcommand{\twocase}[5]{#1 \begin{cases} #2 & \text{#3}\\ #4
&\text{#5} \end{cases}   }

\newcommand{\gep}{\epsilon}

\newcommand{\muST}{\mu_{\rm ST}}
\newcommand{\must}{\mu_{\rm st}}

\numberwithin{equation}{section}

\begin{document}

\title[Effective equidistribution and Sato-Tate for elliptic curves]{Effective equidistribution and the Sato-Tate law for families of elliptic curves}

\author{Steven J. Miller}\email{Steven.J.Miller@williams.edu}
\address{Department of Mathematics and Statistics, Williams College, Williamstown, MA 01267}

\author{M. Ram Murty}\email{murty@mast.queensu.ca}

\address{Department of Mathematics, Queen's University,
Kingston, Ontario, K7L 3N6, Canada}


\thanks{The authors would like to thank Andrew Granville for pointing out the implicit normalization in Birch's work, Igor Shparlinski for discussions on previous results as well as sharing his recent work, and Frederick Strauch for conversations on a hypergeometric proof of Lemma \ref{keycombinatoriallemmaforsuns}; the first named author would also like to thank Cameron and Kayla Miller for quietly sleeping on him while many of the calculations were done. Much of this paper was written when the authors attended the Graduate Workshop on $L$-functions and Random Matrix Theory at Utah Valley University in 2009, and it is a pleasure to thank the organizers. The first named author was partly supported by NSF grant DMS0600848. The second named author was partially supported by an NSERC Discovery grant.}

\maketitle

\begin{abstract} Extending recent work of others, we provide effective bounds on the family of all elliptic curves and one-parameter families of elliptic curves modulo $p$ (for $p$ prime tending to infinity) obeying the Sato-Tate Law. We present two methods of proof. Both use the framework of Murty-Sinha \cite{MS}; the first involves only knowledge of the moments of the Fourier coefficients of the $L$-functions and combinatorics, and saves a logarithm, while the second requires a Sato-Tate law. Our purpose is to illustrate how the caliber of the result depends on the error terms of the inputs and what combinatorics must be done.

\end{abstract}


\section{Introduction}

Recently M. Ram Murty and K. Sinha \cite{MS} proved effective equidistribution results showing the eigenvalues of Hecke operators on the space $S(N,k)$ of cusp forms of weight $k$ and level $N$ agree with the Sato-Tate distribution. Our goal here is to use their framework to prove similar results for families of elliptic curves. We shall do this for the family of all elliptic curves and for one-parameter families of elliptic curves. 


We first review notation and previous results. Let $E: y^2 = x^3 + Ax + B$ with $A, B \in \Z$ be an elliptic curve over $\Q$ with associated $L$-function \be L(E,s)\ =\ \sum_{n=1}^\infty \frac{a_E(n)}{n^s} \ = \ \prod_p \left(1 - \frac{a_E(p)}{p^s} + \frac{\chi_0(p)}{p^{2s-1}}\right)^{-1},\ee where $\Delta=-16(4A^3+27B^2)$ is the discriminant of $E$, $\chi_0$ is the principal character modulo $\Delta$, and \bea a_E(p) & \ =\ & p - \#\{(x,y) \in (\Z/p\Z)^2:  y^2 \equiv x^3 + Ax + B \bmod p\} \nonumber\\ &=& - \sum_{x\bmod p} \js{x^3 + Ax + B}. \eea By Hasse's bound we know $|a_E(p)| \le 2\sqrt{p}$, so we may write $a_E(p) = 2\sqrt{p} \cos \theta_E(p)$, where we may choose $\theta_E(p) \in [0,\pi]$. See \cite{Sil1,Sil2,ST} for more details and proofs of all the needed properties of elliptic curves.

How the $a_E(p)$'s vary is of great interest. One reason for this is that they encode local data (the number of solutions modulo $p$), and are then combined to build the $L$-function, whose properties give global information about $E$. For example, the Birch and Swinnerton-Dyer conjecture \cite{BS-D1,BS-D2} states the order of the group of rational solutions of $E$ equals the order of vanishing of $L(E,s)$ at the central point. While we are far from being able to prove this, the evidence for the conjecture is compelling, especially in the case of complex multiplication and rank at most 1 \cite{Bro,CW,GKZ,GZ,Kol1,Kol2,Ru}. In addition there is much suggestive numerical evidence for the conjecture; for example, for elliptic curves with modest geometric rank $r$, numerical approximations of the first $r-1$ Taylor coefficients are consistent with these coefficients vanishing (see for instance the families studied in \cite{Fe1,Fe2,Mil3}).

If $E$ has complex multiplication\footnote{This means the endomorphism ring is larger than the integers. For example, $y^2 = x^3 - x$ has complex multiplication, as can be seen by sending $(x,y) \to (-x,iy)$. Note $a_E(p) = 0$ if $p \equiv 3 \mod 4$ (this can be seen from the definition of $a_E(p)$ as a sum of Legendre symbols, sending $x\to -x$).} then $a_E(p) = 0$ for half the primes; i.e., $\theta_E(p) = \pi/2$. The remaining angles $\theta_E(p)$ are uniformly distributed in $[0,\pi]$ (this follows from \cite{Deu,He1,He2}).

If $E$ does not have complex multiplication, which is the case for most elliptic curves, then Sato and Tate \cite{Ta} conjectured that as we vary $p$, the distribution of the $\theta_E(p)$'s converges to $2 \sin^2\theta d\theta /\pi$. More precisely, for any interval $I \subset [0,\pi]$ we have \be \lim_{x\to\infty} \frac{\#\{p: p \le x: \theta_E(p) \in I\}}{\#\{p: p \le x\}} \ = \ \int_I \frac{2\sin^2 \theta d\theta}{\pi}; \ee we call $2\sin^2\theta d\theta/\pi$ the Sato-Tate measure, and denote it by $\mu_{{\rm ST}}$. By recent results of Clozel, Harris, Shepherd-Barron and Taylor \cite{CHT,HS-BT,Tay}, this is now known for all such $E$ that have multiplicative reduction at some prime; see also \cite{BZ} for results on the error terms when $|I|$ is small (these results are not for an individual curve, but rather averaged over the family of all elliptic curves) and \cite{B-LGG,B-LGHT} for generalizations to other families of $L$-functions.

Instead of fixing an elliptic curve and letting the prime vary, we can instead fix a prime $p$ and study the distribution of $\theta_E(p)$ as we vary $E$. Before describing our results, we briefly summarize related results in the literature concerning Sato-Tate behavior in families. Serre \cite{Ser} considered a similar question, not for elliptic curves, but rather for $S(N,k)$, the space of cusp forms of weight $k$ on $\Gamma_0(N)$. He proved that for even $k$ with $N+k \to\infty$ the eigenvalues of the normalized $p$\textsuperscript{th} Hecke operators are equidistributed in $[-2,2]$ with respect to the measure \be \mu_p \ = \ \frac{p+1}{\pi} \frac{\sqrt{1-x^2/4}\ dx}{(p^{1/2}+p^{-1/2})^2 - x^2}; \ee changing variables by setting $x = 2\cos\theta$ this is equivalent to the measure $\widetilde{\mu}_p$ on $[0,\pi]$ given by \be \widetilde{\mu}_p \ = \ \frac{2(p+1)}{\pi} \frac{\sin^2 \theta d\theta}{(p^{1/2} + p^{-1/2})^2 - 4\cos^2\theta}. \ee Note that as $p \to \infty$, $\widetilde{\mu}_p \to \mu_{{\rm ST}}$; for $p$ large these two measures assign almost the same probability to an interval $I$, differing by $O(1/p)$. See \cite{CDF,Sar} for other families with a similar distribution.

Serre's theorem was ineffective, and has recently been improved by M. R. Murty and K. Sinha \cite{MS}.
They show that if $\{a_n(p)/p^{(k-1)/2}\}_{1\le i \le \#S(N,k)}$ denote the normalized eigenvalues of the Hecke operator $T_p$ on $S(N,k)$, then \be \frac{\#\{1 \le n \le N: a_n(p)/p^{(k-1)/2} \in I\}}{\#S(N,k)} \ = \ \int_I \mu_p + O\left(\frac{\log p}{\log kN}\right), \ee where $\#S(N,k)$ is the number of cusp forms of weight $k$ and level $N$, and if $N \ge 61$ then by Corollary 15 of \cite{MS} we have\be \frac{3\psi(N)}{200} \ \le \ \#S(N,k) \ \le \ \frac{\psi(N)}{12}+1, \ee where $\psi(N) = N\prod_{p|N}\left(1 + \frac1{p}\right)$. This effective version of equidistribution allows Murty and Sinha to derive many results, such as
\bi \item an effectively computable constant $B_d$ such that
if $J_0(N)$ (the Jacobian of the modular curve $X_0(N)$) is isogenous to a product of $\Q$-simple abelian varieties of dimensions at most $d$, then $N \le B_d$; \item the multiplicity of any given eigenvalue of the Hecke operators is $\ll \frac{s(N,k)\log p}{\log kN}$. \ei

The purpose of this paper is to expand the techniques in \cite{MS} to families of elliptic curves. Unlike \cite{MS,Ser}, we cannot keep the prime fixed throughout the argument, as there are only finitely many distinct reductions of elliptic curves modulo $p$. Instead we fix a prime and study the angles $\theta_E(p)$ for one of the two families below, and then send $p\to\infty$. We study\\

\ben

\item The family of all elliptic curves modulo $p$ for $p \ge 5$. We may write these curves in Weierstrass form as $y^2 = x^3 - ax - b$ with $a,b \in \Z/p\Z$ and $4a^3 \neq 27b^2$. The number of pairs $(a,b)$ satisfying these conditions\footnote{If $a = 0$ then the only $b$ which is eliminated is $b=0$. If $a$ is a non-zero perfect square there are two $b$ that fail, while if $a$ is not a square than no $b$ fail. Thus the number of bad pairs of $(a,b)$ is $p$.} is $p(p-1)$.\\

\item One-parameter families over $\Q(T)$: let $A(T), B(T) \in\Z[T]$ and consider the family $y^2 = x^3 + A(T)x + B(T)$ with non-constant $j(T)$.\footnote{Up to constants, $j(T)$ is $A(T)^3 / (4A(T)^3+27B(T)^2)$.} We specialize $T$ to be a $t\in \Z/p\Z$. The cardinality of the family is $p + O_{A,B}(1)$ (we lose a few values when we specialize as we require the reduced curves to be elliptic curves modulo $p$), where the error is a function of the discriminant of the family. \\

\een


\textbf{Notations:} \bi \item We let $\mathcal{F}_p$ denote either family, and write $V_p$ for its cardinality (which is $p(p-1)$ in the first case and $p + O(1)$ in the second). \\

\item While we may denote the angles by $\theta_E(p)$, $\theta_{a,b}(p)$ or $\theta_t(p)$, as $p$ is fixed for notational convenience and to unify the presentation we shall denote these by $\theta_n$, with $1\le n \le V_p$.\\

\item We let $e(x) = e^{2\pi i x}$.\\

\ei

\ \\

\textbf{Normalizations:}\\

\bi

\item For the family of all elliptic curves, we may match the elliptic curves in pairs $(E,E')$ such that  $\theta_{E'}(p) = \pi - \theta_E(p)$ (and each curve is in exactly one pair); see Remark \ref{rek:matchinpairs} for a proof.  Thus, if we let $x_n = \theta_n(p)/\pi$, we see that the set $\{2x_n\}_{n \le V_p}$ is symmetric about $\pi$. This will be very important later, as it means $\sum_{n \le V_p} \sin(2\pi m x_n) = 0$ for any integer $m$.\\

\item For a one-parameter family of elliptic curves, in general we cannot match the elliptic curves in pairs, and thus the set $\{2\theta_t(p)\}$ is not typically symmetric about $\pi$; see Remark \ref{rek:TateRosenSilverman} for some results about biases in the $\theta_t(p)$'s. This leads to some complications in proving equidistribution, as certain sine terms no longer vanish. To overcome this, following other researchers we consider the technically easier situation where for each elliptic curve we include both $\theta_t(p)$ and $2\pi - \theta_t(p)$. To unify the presentation, instead of normalizing these angles by dividing by $2\pi$ (to obtain a distribution supported on $[0,1]$), we first study the angles modulo $\pi$ and then divide by $\pi$. We thus consider the normalized angles $x_t = \theta_t(p)/\pi$ and $x_{t+V_p} = 1-\theta_t(p)/\pi$ for $1 \le t \le V_p$. Thus we study $2V_p$ normalized angles in $[0,1]$, unlike the case of all elliptic curves where we had $V_p$ angles.\\

\item We set $\widetilde{V}_p = V_p$ for the family of all elliptic curves, and $2V_p$ for a one-parameter family of elliptic curves. We study the distribution of the normalized angles $\{x_n\}_{1\le n \le \widetilde{V}_p}$.\\

\ei

\ \\

\begin{rek}\label{rek:matchinpairs} To see that we may match the angles as claimed for the family of all elliptic curves, consider the elliptic curve $y^2 = x^3-ax-b$ with $4a^3\neq 27b^2$. Let $c$ be any non-residue modulo $p$, and consider the curve $y^2 = x^3 - ac^2x - bc^3$. Using the Legendre sum expressions for $a_E(p)$ and $a_{E'}(p)$, using the automorphism $x \to cx$ we see the second equals $\js{c}$ times the first; as we have chosen $c$ to be a non-residue, this means $2\sqrt{p}\cos(\theta_{E'}(p)) = -2\sqrt{p}\cos(\theta_E(p))$, or $\theta_{E'}(p)=\pi-\theta_E(p)$ as claimed.
\end{rek}

\begin{rek}\label{rek:TateRosenSilverman} If the one-parameter family of elliptic curves has rank $r$ over $\Q(T)$ and satisfies Tate's conjecture (see \cite{Ta,RS}), then Rosen and Silverman \cite{RS} prove a conjecture of Nagao \cite{Na}, which states \be \lim_{X \to \infty} -\frac{1}{X} \sum_{p \leq X} \frac{A_1(p)\log p}{p}
\ \ =\  r \ee where $A_1(p) := \sum_{t\bmod p} a_t(p)$. Tate's conjecture is known for rational surfaces.\footnote{An elliptic surface $y^2 = x^3 + A(T)x + B(T)$ is rational if and only if one of the following is true: $(1)\ $ $0 < \max\{3 {\rm deg} A, 2{\rm deg} B\} < 12;$ $(2)\ $ $3{\rm deg} A = 2{\rm deg} B = 12$ and $\mbox{ord}_{t=0}t^{12} \Delta(t^{-1})$ $=$ $0$.} This bias has been used by S. Arms, \'{A}. Lozano-Robledo and S. J. Miller \cite{AL-RM} to construct one-parameter families with moderate rank by finding families where $A(p)$ is essentially $-rp$. As there are about $p$ curves modulo $p$, this represents a bias of about $-r$ on average per curve; as each $a_t(p)$ is of order $\sqrt{p}$, we see in the limit that this bias should be quite small per curve (though significant enough to lead to rank, it gives a lower order contribution to the distribution for each prime, and will be dwarfed by our other errors).
\end{rek}



Our goal is to prove effective theorems on the rate of convergence as $p\to\infty$ to the Sato-Tate measure, which requires us to obtain effective estimates for \be \left| \#\{n \le \widetilde{V}_p: \theta_n \in I\} - \muST(I) \widetilde{V}_p\right|. \ee Here $\muST$ is the Sato-Tate measure on $[0,\pi]$ given by \be \muST(T) \ = \ \int_I \frac{2}{\pi} \sin^2 t dt \ \ \ \ I \subset [0,\pi], \ee and for $n \le V_p$, $2\sqrt{p}\cos(\theta_n)$ is the number of solutions modulo $p$ of the elliptic curve $E_n: y^2 = x^3 + a_nx + b_n$. Equivalently, using the normalization $x_n=\theta_n/\pi$ to obtain a distribution on $[0,1]$, the Sato-Tate measure become \be \must(I) \ = \ \int_I 2 \sin^2(\pi x)dx, \ \ \ I \subset [0,1]. \ee For a sequence of numbers $x_n$ modulo 1, a measure $\mu$ and an interval $I \subset [0,1]$, let \bea N_I(\widetilde{V}_p) & \ = \ & \#\{n \le \widetilde{V}_p: x_n \in I\} \nonumber\\ \mu(I) & \ =\ & \int_I \mu(t) dt. \eea The discrepancy $D_{I,\widetilde{V}_p}(\mu)$ is \bea D_{I,\widetilde{V}_p}(\mu) & \ =\  & \left|N_I(\widetilde{V}_p) - \widetilde{V}_p \mu(I)\right|; \eea with this normalization, the goal is to obtain the best possible estimate for how rapidly $D_{I,\widetilde{V}_p}(\mu)/\widetilde{V}_p$ tends to 0. 

Previous work has obtained a power savings in convergence to Sato-Tate for two-parameter families of elliptic curves (such as the entire family of all elliptic curves, or parametrizations such as $y^2 = x^3 + f(a)x + g(b)$ with $a$ and $b$ varying in appropriate ranges); see the papers by Banks and Shparlinski \cite{BS,Sh1,Sh2} for saving $\widetilde{V}_p^{1/4}$ in Sato-Tate convergence. The key step in these arguments is \be\label{eq:resultfromKatz} \frac{1}{(p-1)^2} \sum_{a,b \bmod p \atop 4a^3 + 27b^2 \not\equiv 0 \bmod p} \frac{\sin((k+1)\theta_{a,b}(p))}{\sin(\theta_{a,b}(p)} \ \ll \ kp^{-1/2}, \ \ \ k \ = \ 1, 2, \dots; \ee see Theorem 13.5.3 from \cite{Ka} for a proof. One can obtain new and similar results for one-parameter families of elliptic curves by appealing to a result of Michel \cite{Mic}, which we do in \S\ref{sec:effequidoneparam}. Our main results are the following.

\begin{thm}[Family of all curves]\label{thm:allellipticcurves} For the family of all elliptic curves modulo $p$, as $p\to\infty$ we have \be D_{I,\widetilde{V}_p}(\must) \ \le\ C \frac{\widetilde{V}_p}{\log \widetilde{V}_p} \ee for some computable $C$.  Note that in this family, $\widetilde{V}_p = V_p$ and for each curve we include one normalized angle, $x_n = \theta_n/\pi \in [0,1]$. \end{thm}

\begin{thm}[One-parameter family of elliptic curves]\label{thm:oneparamellipticcurves} For a one-parameter family of elliptic curves over $\Q(T)$ with non-constant $j$-invariant, we have \be D_{I,\widetilde{V}_p}(\must) \ \le \ C \widetilde{V}_p^{3/4} \ee for some computable $C$. Note that in this family, $\widetilde{V}_p = 2V_p$ and for each curve we include two normalized angles, $x_n = \theta_n/\pi$ and $x_{n+V_p} = 1-\theta_n/\pi$, with $\theta_n \in [0,\pi]$. \end{thm}

\ \\

Stronger results than Theorem \ref{thm:allellipticcurves} are known; as remarked above, convergence to Sato-Tate with an error of size $\widetilde{V}_p^{3/4}$ instead of $\widetilde{V}_p / \log \widetilde{V}_p$ is obtained in \cite{BS,Sh1,Sh2}. We present these weaker arguments to highlight how one may attack these problems possessing only knowledge of the moments, and not the functions of the angles, in the hope that these arguments might be of use to other researchers attacking similar questions where we only have formulas for the moments of the coefficients. We will thus illustrate the effectiveness (in both senses of the word) of the techniques in \cite{MS}, as well as illustrate the loss of information that comes from having to trivially bound certain combinatorial sums. As we have not found similar effective results in the literature for one-parameter families, in order to get the best possible results we do not use formulas for the moments but rather estimates for the analogue of \eqref{eq:resultfromKatz}. It is worth remarking that we can recover the results of \cite{BS,Sh1,Sh2} by our generalization of \cite{MS} provided we also use \eqref{eq:resultfromKatz} (see \cite{Ka}) instead of results from Birch \cite{Bi} on moments; this shows the value of the formulation in \cite{MS}.

We summarize the key ingredients of the proofs, and discuss why the second result has a much better error term than the first. Similar to \cite{MS}, both theorems follow from an analysis of $\sum_{n\le \widetilde{V}_p} e(mx_n)$ (we use $x_n = \theta_n/\pi$ in order to have a distribution supported on $[0,1]$). For the family of all elliptic curves, after some algebra we see this is equivalent to understanding $\sum_{n\le \widetilde{V}_p} \cos (2m\theta_n)$; using a combinatorial identity (see \cite{Mil4}) this is equivalent to a linear combination of sums of the form  $\sum_{n\le \widetilde{V}_p} (\cos \theta_n)^{2r}$. These sums are essentially the $2r$\textsuperscript{th} moments of the Fourier coefficients of the family of all elliptic curves modulo $p$. Birch \cite{Bi} evaluated these, and showed the answers are the Catalan numbers\footnote{The Catalan numbers are the moments of the semi-circle distribution, which is related to the Sato-Tate distribution through a simple change of variables.} plus lower order terms. Our equidistribution result then follows from a combinatorial identity of a sum of weighted Catalan numbers; our error term is poor due to the necessity of losing cancelation in bounding the contribution from the sums of the error terms. 

The proof of Theorem \ref{thm:oneparamellipticcurves} is easier, as now instead of inputting results on the moments we instead use a result of Michel \cite{Mic} for the sum over the family of ${\rm sym}_k(\theta_n)=  \sin((k+1)\theta_n) / \sin \theta_n$. This is easily related to our quantity of interest, $\cos (2m\theta_n)$, through identities of Chebyshev polynomials: \be \cos (2m\theta_n) \ =\ \foh {\rm sym}_{2m}(\theta_n) - \foh{\rm sym}_{2m-2}(\theta_n).\ee The advantage of having a formula for the quantity we want and not a related quantity is that we avoid trivially estimating the errors in the combinatorial sums. These calculations increased the size of the error significantly, and this is why Theorem \ref{thm:oneparamellipticcurves} is stronger than Theorem \ref{thm:allellipticcurves}, though the error term in Theorem \ref{thm:allellipticcurves} is comparable to the error terms of the equivalent quantities in \cite{MS} for the family of cuspidal newforms. Michel proves his result by using a cohomological interpretation, and this results in the error term being $p^{-1/2}$ smaller than the main term; it is this savings in the quantity we are directly interested in that leads to the superior error estimates.

The paper is organized as follows. After reviewing the needed results from Murty-Sinha \cite{MS} in \S\ref{sec:preliminaries}, we prove Theorem \ref{thm:allellipticcurves} in \S\ref{sec:equidallcurves} and Theorem \ref{thm:oneparamellipticcurves} in \S\ref{sec:effequidoneparam}. For completeness the needed combinatorial identities are proved in Appendix \ref{sec:combidentities}, and in Appendix \ref{sec:momentsallcurves} we correct some errors in explicit formulas for moments in Birch's paper \cite{Bi} (where he neglected to mention that his sums are normalized by dividing by $p-1$).


\section{Effective Equidistribution Preliminaries}\label{sec:preliminaries}

We quickly review some needed results from Murty-Sinha \cite{MS}; while our setting is similar to the problems they investigated, there are slight differences which require generalizations of some of their results. Assume $\mu = F(-x)dx$ with \be F(x)\ = \ \sum_{m=-\infty}^\infty c_m e(mx) \ee where $e(z) = \exp(2\pi i z)$. Theorem 8 from \cite{MS} is

\begin{thm}\label{thm:MSmain} Let $\{x_n\}$ be a sequence of real numbers in $[0,1]$ and let the notation be as above. Assume for each $m$ that \be \lim_{\widetilde{V}_p\to \infty} \frac1{\widetilde{V}_p} \sum_{n \le \widetilde{V}_p} e(mx_n) \ = \ c_m \ \ \ {\rm and} \ \ \ \sum_{m=-\infty}^\infty |c_m| \ < \ \infty. \ee Let $||\mu|| = \sup_{x \in [0,1]} |F(x)|$ with $\mu = F(-x)dx$. Then the discrepancy satisfies \bea\label{eq:definitiondiscdistanceMS} & & D_{I,\widetilde{V}_p}(\mu)  \ \le \  \frac{\widetilde{V}_p||\mu||}{M+1} \nonumber\\ & & + \sum_{1 \le m \le M} \left(\frac1{M+1} + \min\left(b-a,\frac{1}{\pi|m|}\right)\right) \left|\sum_{n=1}^{\widetilde{V}_p} e(mx_n) - \widetilde{V}_p c_m\right| \nonumber\\ \eea for any natural numbers $\widetilde{V}_p$ and $M$. \end{thm}

Unfortunately, Theorem \ref{thm:MSmain} is not directly applicable in our case. The reason is that there we have a limit as $\widetilde{V}_p\to\infty$ in the definition of the $c_m$, where for us we fix a prime $p$ and have $\widetilde{V}_p = p(p-1)$ for the family of all elliptic curves curves modulo $p$, or $p + O(1)$ for a one-parameter family. Analyzing the proof of Theorem 8 from \cite{MS}, however, we see that the claim holds for \emph{any} sequence $c_m$ (obviously if $\widetilde{V}_p^{-1} \sum_{n \le \widetilde{V}_p} e(mx_n)$ is not close to $c_m$ then the discrepancy is large). We thus obtain

\begin{thm}\label{thm:MSmainmodified} Let $\{x_n\}$ be a sequence of real numbers in $[0,1]$ and let the notation be as above. Let $\{c_m\}$ be a sequence of numbers such that $\sum_{m=-\infty}^\infty |c_m|  <  \infty$ (we will take $c_0=1$, $c_{\pm 1} = -1/2$ and all other $c_m$'s equal to zero). Let $||\mu|| = \sup_{x \in [0,1]} |F(x)|$ with $\mu = F(-x)dx$.  Then the discrepancy satisfies \bea\label{eq:definitiondiscdistanceUS} & & D_{I,\widetilde{V}_p}(\mu)  \ \le \  \frac{\widetilde{V}_p||\mu||}{M+1} \nonumber\\ & & + \sum_{1 \le m \le M} \left(\frac1{M+1} + \min\left(b-a,\frac{1}{\pi|m|}\right)\right) \left|\sum_{n=1}^{\widetilde{V}_p} e(mx_n) - \widetilde{V}_p c_m\right| \nonumber\\ \eea for any natural numbers $\widetilde{V}_p$ and $M$. \end{thm}


To simplify applying the results from \cite{MS}, we study the normalized angles $x_n$. Under our normalization, the Sato-Tate measure becomes \be \must(I) \ = \ \int_I 2 \sin^2(\pi x)dx, \ \ \ I \subset [0,1]. \ee The Fourier coefficients of $\must$ are readily calculated.

\begin{lem} Let $\must = F(-x)dx$ be the normalized Sato-Tate distribution on $[0,1]$ with density $2\sin^2(\pi x)$. We have \be F(x) \ = \ 1 - \frac12 \left(e(x) + e(-x)\right), \ee which implies that the Fourier coefficients are $c_0 = 1$, $c_{\pm 1} = -1/2$ and $c_m = 0$ for $|m| \ge 2$. \end{lem}

\begin{proof} The proof is immediate from the expansion of $F$ as a sum of exponentials, which follows from the identities $\cos(2\theta) = 1 - 2\sin^2(\theta)$ and $e(\theta) = \cos(2\pi\theta) + i\sin(2\pi\theta)$.
\end{proof}


\section{Proof of Effective Equidistribution for All Curves}\label{sec:equidallcurves}

We use Birch's \cite{Bi} results on the moments of the family of all elliptic curves modulo $p$ (there are some typos in his explicit formulas; we correct these in Appendix \ref{sec:momentsallcurves}); unfortunately, these are results for quantities such as $(2 \sqrt{p} \cos \theta_n)^{2R}$, and the quantity which naturally arises in our investigation is $e(mx_n)$ (with $x_n$ running over the normalized angles $\theta_{a,b}(p)/\pi$), specifically \be \left| \sum_{n = 1}^{\widetilde{V}_p} e(mx_n) - \widetilde{V}_p c_m \right|. \ee By applying some combinatorial identities we are able to rewrite our sum in terms of the moments, which allows us to use Birch's results. The point of this section is not to obtain the best possible error term (which following \cite{BS,Sh1,Sh2} could be obtained by replacing Birch's bounds with \eqref{eq:resultfromKatz}) but rather to highlight how one may generalize and apply the framework from \cite{MS}.

We first set some notation. Let $\sigma_k(T_p)$ denote the trace of the Hecke operator $T_p$ acting on the space of cusp forms of dimension $-2k$ on the full modular group. We have $\sigma_{k+1}(T_p) = O(p^{k + c + \gep})$, where from \cite{Sel} we see we may take $c = 3/4$ (there is no need to use the optimal $c$, as our final result, namely \eqref{eq:Mgivenc}, will yield the same order of magnitude result for $c=3/4$ or $c=0$). Let $\mathcal{M}_p(2R)$ denote the $2R$\textsuperscript{th} moment of $2\cos(\theta_n) = 2\cos(\pi x_n)$ (as we are concerned with the normalized values, we use slightly different notation than in \cite{Bi}): \be \mathcal{M}_p(2R) \ = \ \frac1{\widetilde{V}_p} \sum_{n=1}^{\widetilde{V}_p} \left(2 \cos(\pi  x_n)\right)^{2R}. \ee

\begin{lem}[Birch]\label{lem:birch} Notation as above, we have \be \mathcal{M}_p(2R) \ = \ \frac1{R+1}\ncr{2R}{R} + O\left(2^{2R} \widetilde{V}_p^{-\frac{1-c-\gep}2}\right);\ee we may take $c = 3/4$ and thus there is a power saving.\footnote{Note $\frac1{R+1}\ncr{2R}{R}$ is the $R$\textsuperscript{th} Catalan number. The Catalan numbers are the moments of the semi-circle distribution, which is related to the Sato-Tate distribution by a simple change of variables.}
\end{lem}

\begin{proof} The result follows from dividing the equation for $S_R^\ast(p)$ on the bottom of page 59 of \cite{Bi} by $p^R$, as we are looking at the moments of the normalized Fourier coefficients of the elliptic curves, and then using the bound $\sigma_{k+1}(T_p) = O(p^{k+c+\gep})$, with $c = 3/4$ admissible by \cite{Sel}. Recall $\widetilde{V}_p=p(p-1)$ is the cardinality of the family. We have \bea \mathcal{M}_p(2R) & \ = \ & \frac1{R+1}\ncr{2R}{R} \frac{p(p-1)}{\widetilde{V}_p} \nonumber\\ & & \ \ + \ O\left(\sum_{k=1}^R \frac{2k+1}{R+k+1} \ncr{2R}{R+k} \frac{p^{1+c+\gep}}{\widetilde{V}_p} + \frac{p}{p^R \widetilde{V}_p}\right) \nonumber\\ & = & \frac1{R+1}\ncr{2R}{R} + O\left(2^{2R} \widetilde{V}_p^{-\frac{1-c-\gep}2}\right)\eea since $\widetilde{V_p} = p(p-1)$.
\end{proof}

A simple argument (see Remark \ref{rek:matchinpairs}) shows that the normalized angles are symmetric about $1/2$. This implies \be \sum_{n=1}^{\widetilde{V}_p} e(mx_n) \ = \ \sum_{n=1}^{\widetilde{V}_p} \cos(2\pi mx_n) + i  \sum_{n=1}^{\widetilde{V}_p} \sin(2\pi mx_n) \ = \ \sum_{n=1}^{\widetilde{V}_p} \cos(2m\theta_n),\ee where the sine piece does not contribute as the angles are symmetric about $1/2$, and we are denoting the $\widetilde{V}_p$ non-normalized angles by $\theta_n$.

Thus it suffices to show we have a power saving in \be \left| \sum_{n=1}^{\widetilde{V}_p} \cos(2m\theta_n) - \widetilde{V}_p c_m\right|. \ee By symmetry, it suffices to consider $m \ge 0$.

\begin{lem}Let $c_0 = 1$, $c_{\pm 1} = -1/2$ and $c_m = 0$ otherwise. There is some $c < 1$ such that \be \left| \sum_{n=1}^{\widetilde{V}_p} \cos(2m\theta_n) - \widetilde{V}_p c_m\right| \ \ll \ \left(m^2 2^{3m} \widetilde{V}_p^{-\frac{1-c-\gep}2}\right); \ee by the work of Selberg \cite{Sel} we may take $c = 3/4$.
\end{lem}

\begin{proof}
The case $m=0$ is trivial. For $m=1$ we use the trigonometric identity $\cos(2\theta_n) = 2 \cos^2(\theta_n) - 1$. As $c_{\pm 1} = -1/2$ we have \bea \sum_{n=1}^{\widetilde{V}_p} \cos(2\theta_n) - \frac{\widetilde{V}_p}2 & \ = \ & \sum_{n=1}^{\widetilde{V}_p} \left[ \left(2\cos^2 \theta_n - 1\right) + \frac12\right] \nonumber\\ & = & \frac12 \sum_{n=1}^{\widetilde{V}_p} \left((2 \cos \theta_n)^2 - 1\right) \nonumber\\ & = & \frac12 \sum_{n=1}^{\widetilde{V}_p} \left(\frac{(2\sqrt{p}\cos\theta_n)^2}{p} - 1\right). \eea Note the sum of $(2\sqrt{p}\cos \theta_n)^2$ is the second moment of the number of solutions modulo $p$. From \cite{Bi} we have that this is $p + O(1)$; the explicit formula given in \cite{Bi} for the second moment is wrong; see Appendix \ref{sec:momentsallcurves} for the correct statement. Substituting yields \bea \left|\sum_{n=1}^{\widetilde{V}_p} \cos(2\theta_n) - \frac{\widetilde{V}_p}2\right| & \ \ll \ & O(1). \eea

The proof is completed by showing that $\sum_{n=1}^{\widetilde{V}_p} \cos(2m\theta_n) = O_m(\widetilde{V}_p^{1/2})$ provided $2 \le m \le M$. In order to obtain the best possible results, it is important to understand the implied constants, as $M$ will have to grow with $\widetilde{V}_p$ (which is of size $p^2$). While it is possible to analyze this sum for any $m$ by brute force, we must have $M$ growing with $p$, and thus we need an argument that works in general. As $c_{\pm 1} \neq 0$ but $c_m = 0$ for $|m| \ge 2$, we expect (and we will see) that the argument below does break down when $|m| = 1$.

There are many possible combinatorial identities we can use to express $\cos(2m\theta_n)$ in terms of powers of $\cos(\theta_n)$. We use the following (for a proof, see Definition 2 and equation (3.1) of \cite{Mil4}): \be\label{eq:sjmexpansiontwocos2mtheta} 2\cos(2m\theta_n) \ = \ \sum_{r=0}^{m} c_{2m,2r} (2 \cos \theta_n)^{2r}, \ee where $c_{2r} = (2r)!/2$, $c_{0,0} = 0$, $c_{2m,0} = (-1)^m 2$ for $m\ge 1$, and for $1 \le r \le m$ set \be c_{2m,2r} \ = \ \frac{(-1)^{r+m}}{c_{2r}}\prod_{j=0}^{r-1} (m^2 - j^2) \ = \  \frac{(-1)^{m+r}}{c_{2r}}\frac{m \cdot (m+r-1)!}{(m-r)!}. \ee We now sum \eqref{eq:sjmexpansiontwocos2mtheta} over $n$ and divide by $\widetilde{V}_p$, the cardinality of the family. In the argument below, at one point we replace $2^{2r}$ in an error term with $2012\frac1{r+1}\ncr{2r}{r} \cdot m^2$; this allows us to pull the $r$\textsuperscript{th} Catalan number, $\frac{1}{r+1}\ncr{2r}{r}$, out of the error term.\footnote{The reason this is valid is that the largest binomial coefficient is the middle (or the middle two when the upper argument is odd). Thus $2^{2r} = (1+1)^{2r} \le (2r+1) \ncr{2r}{r} \le 2(m+1)\ncr{2r}{r}$ (as $m \le r$), and the claim follows from $\frac{2012m^2}{r+1} \ge 2(m+1)$ for $m \ge 2$ and $0 \le r \le m$.} Using
Lemma \ref{lem:birch} we find \bea & & \frac1{\widetilde{V}_p} \sum_{n=1}^{\widetilde{V}_p} 2\cos(2m\theta_n)  \ = \  \sum_{r=0}^m c_{2m,2r} \frac1{\widetilde{V}_p} \sum_{n=1}^{\widetilde{V}_p} (2 \cos \theta_n)^{2r} \nonumber\\ & = & \sum_{r=0}^m \left( \frac1{r+1}\ncr{2r}{r} + O\left(2^{2r} \widetilde{V}_p^{-\frac{1-c-\gep}2}\right) \right) c_{2m,2r}\nonumber\\ &=& \sum_{r=0}^m \left( \frac1{r+1}\frac{(2r)!}{r!r!} \frac{(-1)^{m+r} 2}{(2r)!}  \frac{m \cdot (m+r)!}{(m-r)! \cdot (m+r)} \right)\nonumber\\ & & \ \ \ \cdot \ \left(1 + O\left(m^2 \widetilde{V}_p^{-\frac{1-c-\gep}2} \right)\right) \nonumber\\ & = & (-1)^m 2m \sum_{r=0}^m \left((-1)^r\frac{m!}{r!(m-r)!}   \frac{(m+r)!}{m!r!}\frac{1}{(r+1)(m+r)} \right)\nonumber\\ & & \ \ \ \cdot \ \left(1 + O\left(m^2 \widetilde{V}_p^{-\frac{1-c-\gep}2} \right)\right)  \nonumber\\ & = & (-1)^m 2m \sum_{r=0}^m \left((-1)^r\ncr{m}{r}  \ncr{m+r}{r}\frac{1}{(r+1)(m+r)} \right)\nonumber\\ & & \ \ \ \cdot \ \left(1 + O\left(m^2 \widetilde{V}_p^{-\frac{1-c-\gep}2} \right)\right). \eea We first bound the error term. For our range of $r$, $\ncr{m+r}{r} \le \ncr{2m}{m} \le 2^{2m}$. The sum of $\ncr{m}{r}$ over $r$ is $2^m$, and we get to divide by at least $m+r \ge m$. Thus the error term is bounded by \be O\left(m^2 2^{3m} \widetilde{V}_p^{-\frac{1-c-\gep}2} \right). \ee We now turn to the main term. It it just $(-1)^m 2m$ times the sum in Lemma \ref{keycombinatoriallemmaforsuns}, which is shown in that lemma to equal 0 for any $|m| \ge 2$.
\end{proof}

\begin{rek} Without Lemma \ref{keycombinatoriallemmaforsuns}, our combinatorial expansion would be useless. We thus give several proofs in the appendix (including a brute force, hypergeometric and an application of Zeilberger's Fast Algorithm). \end{rek}

\begin{rek} It is possible to get a better estimate for the error term by a more detailed analysis of $\sum_{r\le m} \ncr{m}{r} \ncr{m+r}{r}$; however, the improved estimates only change the constants in the discrepancy estimates, and not the savings. This is because this sum is at least as large as the term when $r \approx m/2$, and this term contributes something of the order $3^{3m/2} / m$ by Stirling's formula. We will see that any error term of size $3^{am}$ for a fixed $a$ gives roughly the same value for the best cutoff choice for $M$, differing only by constants. Thus we do not bother giving a more detailed analysis to optimize the error here. \end{rek}

We now prove the first of our two main theorems.

\begin{proof}[Proof of Theorem \ref{thm:allellipticcurves}] We must determine the optimal $M$ to use in \eqref{eq:definitiondiscdistanceUS}: \bea  D_{I,\widetilde{V}_p}(\must) & \ \ll \ & \frac{\widetilde{V}_p}{M+1}  + \sum_{1 \le m \le M} \left(\frac1{M+1} + \frac1{m} \right) \left(m^2 2^{3m} \widetilde{V}_p^{-\frac{1-c-\gep}2}\right) \nonumber\\ & \ll & \frac{\widetilde{V}_p}{M} +  M 2^{3M} \widetilde{V}_p^{-\frac{1-c-\gep}2} \nonumber\\ \eea as $\frac{1}{M+1} \ll \frac1{m}$ and $\sum_{m \le m} 2^{3m} \ll 2^{3M}$. For all $c > 0$ we find the minimum error by setting the two terms equal to each other, which yields \be \widetilde{V}_p^{\frac{3-c-\gep}{2}} \ = \ M^2 2^{3M} \ \ll \ e^{3M}, \ee which when equating yields\footnote{We could obtain a slightly better constant below with a little more work; however, as it will not affect the quality of our result we prefer to give the simpler argument with a slightly worse constant.} \be e^{3M} \ \approx \ e^{\frac{3-c-\gep}{2} \log \widetilde{V}_p}, \ee which implies \be\label{eq:Mgivenc} M \ \approx \ \frac{3-c-\gep}{6} \log \widetilde{V}_p. \ee We thus see that we may find a constant $C$ such that \be D_{I,\widetilde{V}_p}(\must) \ \le \ C \frac{\widetilde{V}_p}{\log \widetilde{V}_p}.\ee
\end{proof}



\section{Proof of Effective Equidistribution for One-parameter families}\label{sec:effequidoneparam}

Instead of studying the family of all elliptic curves, we can also investigate one-parameter families over $\Q(T)$. Thus, consider the family $\mathcal{E}: y^2 = x^3 + A(T) x + B(T)$, where $A(T)$ and $B(T)$ are in $\Z(T)$. We assume that $j(T)$ is not constant for the family. Michel \cite{Mic} proved a Sato-Tate law for such families. In particular, he proved

\begin{thm}[Michel \cite{Mic}]\label{thm:michel} Consider a one-parameter family of elliptic curves over $\Q(T)$ with non-constant $j$-invariant. Let $c_\Delta$ denote the number of complex zeros of $\Delta(z) = 0$ (where $\Delta$ is the discriminant), $\psi_p$ an additive character (and set $\delta_{\psi_p} = 0$ if this character is trivial and 1 otherwise), and write $a_{t,p}$ as $2\sqrt{p} \cos \theta_{t,p}$ with $\theta_{t,p} \in [0,\pi]$. Let \be {\rm sym}_k(\theta) \ = \ \frac{\sin((k+1)\theta)}{\sin \theta}. \ee
Then \be \left| \frac1{p} \sum_{t \bmod p \atop \Delta(t) \neq 0} {\rm sym}_k\theta_{t,p}\right| \ \le \ \frac{(k+1) (c_\Delta - \delta_{\psi_p} - 1)}{\sqrt{p}}. \ee Additionally, we have \be \left| \frac1{p} \sum_{t \bmod p \atop \Delta(t) \neq 0} \cos \theta_{t,p} \right| \ \le \ \frac{C}{\sqrt{p}} \ee for some $C$ depending on the family. Finally, we may drop the additive character and drop the restriction that $\Delta(t) \neq 0$ at the cost of a bounded number of summands, each of which is at most $(k+1)$,\footnote{This is readily seen by writing $\sin ((k+1)\theta) = \sin (\theta) \cos (k\theta) + \cos (\theta)\sin (k\theta)$ and proceeding by induction.} which implies these relations still hold provided we multiply the bounds on the right hand side by some constant $C'$.
\end{thm}

\begin{rek}\label{rek:michelsharp} Miller \cite{Mil2} showed that the error term in Theorem \ref{thm:michel} is sharp. Specifically, the second moment of the family $y^2 = x^3+Tx^2+1$ of elliptic curves over $\Q(T)$ for $p>2$ is \be
A_{2}(p) \ := \ \sum_{t\bmod p} a_t(p)^2 \ = \  p^2 -
n_{3,2,p}p - 1 + p\sum_{x \bmod p} \js{4x^3+1}, \ee where
$n_{3,2,p}$ denotes the number of cube roots of $2$ modulo $p$. For
any $[a,b] \subset [-2,2]$ there are infinitely many primes $p
\equiv 1 \bmod 3$ such that \be A_{2}(p) - \left(p^2-
n_{3,2,p}p - 1\right)\ \in\ [a\cdot p^{3/2},b\cdot p^{3/2}].\ee \end{rek}

Theorem \ref{thm:michel} is used by Michel to obtain good estimates for the average rank in these families, as well as (of course) proving Sato-Tate laws. Using our techniques above, we can convert Michel's bounds to a quantified equidistribution law.

We recall the notation for Theorem \ref{thm:oneparamellipticcurves}. Consider a one-parameter family of elliptic curves over $\Q(T)$ with non-constant $j(T)$. Let there be $V_p = p + O(1)$ reduced curves modulo $p$, and set $\widetilde{V}_p = 2V_p$. For each curve $E_t$ consider the angles $\theta_{t,p}$ and $\pi-\theta_{t,p}$, with $\theta_{t,p} \in [0,1]$, and the normalized angles $x_n = \theta_{t,p}/\pi$ and $x_{n+V_p}=1-\theta_{t,p}/\pi$ (for $1\le n \le V_p$).

\begin{proof}[Proof of Theorem \ref{thm:oneparamellipticcurves}] We must show $D_{I,\widetilde{V}_p}(\must) \ \ll \ \widetilde{V}_p^{3/4}$ (where $\widetilde{V}_p \approx 2p$). As in the proof of Theorem \ref{thm:allellipticcurves}, it suffices to show \be \left| \sum_{t \bmod p} \cos(2m\theta_{t,p}) - c_{m} p\right|, \ee with $c_0 = 1$, $c_1 = -1/2$ and all other $c_{m} = 0$. This is because we have enlarged our set of normalized angles to be symmetric about 1/2. Thus when we study $e(mx_n) = \cos(2\pi m x_n) + i\sin(2\pi m x_n)$, the sine sum vanishes. We are therefore left with the cosine sum, with the normalized angles $x_n$ and $x_{n+V_p}$ contributing equally. Thus we may replace the sum of the cosine piece over $n$ with a sum over the angles $\theta_{t,p}$, so long as we remember to multiply by 2 when computing the discrepancy later. While we should subtract $c_m V_p $ and not $c_m p$, as $V_p = p+O(1)$ the error in doing this is dwarfed by the error of the piece we are studying.

The case of $2m=0$ is trivial. If $2m=2$, then we are studying $\cos 2 \theta_{t,p} = -\foh + \foh {\rm sym}_2(\theta)$. By Theorem \ref{thm:michel}, we thus find that \bea \left| \sum_{t \bmod p} \cos(2\theta_{t,p}) + \frac{p}{2}\right| & \ = \  & \left| \sum_{t \bmod p}  \foh {\rm sym}_2(\theta) \right| \ \le \ \frac{C}{\sqrt{p}}. \eea For higher $m$, we use Chebyshev polynomials (see \cite{Wi}). The Chebyshev polynomials of the first kind are given by $T_\ell(\cos \theta) = \cos(\ell\theta)$; the Chebyshev polynomials of the second kind are $U_\ell(\cos \theta) = {\rm sym}_{\ell+1}(\theta)$. These polynomials are related by \be T_\ell(\cos \theta) \ = \ \frac{U_\ell(\cos \theta) - U_{\ell-2}(\cos \theta)}{2} \ = \ \frac{{\rm sym}_\ell(\theta) - {\rm sym}_{\ell-2}(\theta)}{2}; \ee we use this with $\ell = 2m \ge 4$. Using Theorem \ref{thm:michel} we see that for $m \ge 2$, \be \left| \sum_{t \bmod p} \cos(2m\theta_{t,p}) \right| \ \le \ Cm\sqrt{p}. \ee

From \eqref{eq:definitiondiscdistanceUS}, the discrepancy satisfies \bea & &\foh D_{I,\widetilde{V}_p}(\must)  \ \le \  \frac{p||\mu||}{M+1} \nonumber\\ & & + \sum_{1 \le m \le M} \left(\frac1{M+1} + \min\left(b-a,\frac{1}{\pi|m|}\right)\right) \left|\sum_{t=1}^p e(mx_n) - c_mp\right|.\nonumber\\ \eea
Using our bounds, we have \bea D_{I,p}(\mu) & \ \ll \ & \frac{p||\mu||}{M+1} + \sum_{m=1}^M \frac{Cm\sqrt{p}}{m} \ \ll \ \frac{p}{M} + M\sqrt{p}. \eea The two error terms are of the same order of magnitude when $M^2 = \sqrt{p}$, or $M = p^{1/4}$. This leads to \be D_{I,p}(\mu) \ \ll \ p^{3/4}, \ee which should be compared to a discrepancy of order $p$; in other words, we have a power savings (much better than the logarithmic savings in the family of all elliptic curves).
\end{proof}

\begin{rek} Note we could have used the Chebyshev identities to handle the $m=1$ case as well, as in fact we implicitly did when we rewrote $\cos 2\theta$; we prefer to break the analysis into two cases as the $m=1$ case has $c_m \neq 0$. \end{rek}

\begin{rek} Rosen and Silverman \cite{RS} proved a conjecture of Nagao \cite{Na} relating the distribution of the $a_E(p)$'s and the rank. Unfortunately the known lower order term due to the rank of the family is of size $p^{1/2}$, which is significantly smaller than the error terms of size $p^{3/4}$ analyzed above. As noted in Remark \ref{rek:michelsharp}, the error term is sharp and cannot be improved for all families.  \end{rek}


\appendix


\section{Combinatorial Identities}\label{sec:combidentities}

We first state some needed properties of the binomial coefficients. For $n, r$ non-negative integers we set $\ncr{n}{k} = \frac{n!}{k!(n-k)!}$. We generalize to real $n$ and $k$ a positive integer by setting \be \ncr{n}{k} \ = \ \frac{n(n-1)\cdots (n-(k-1))}{k!}, \ee which clearly agrees with our original definition for $n$ a positive integer. Finally, we set $\ncr{n}{0} = 1$ and $\ncr{n}{k} = 0$ if $k$ is a negative integer.

To prove our main result we need the following two lemmas; we follow the proofs in \cite{Ward}.

\begin{lem}[Vandermonde's Convolution Lemma] Let $r,s$ be any two real numbers and $k, m, n$ integers. Then \be \sum_k \ncr{r}{m+k} \ncr{s}{n-k} \ = \ \ncr{r+s}{m+n}. \ee
\end{lem}

\begin{proof} It suffices to prove the claim when $r, s$ are integers. The reason is that both sides are polynomials, and if the polynomials agree for an infinitude of integers then they must be identical. It suffices to consider the special case $m=0$, in which case we are reduced to showing \be\label{eq:ncrcombidentityvandermonde} \ncr{r}{k} \ncr{s}{n-k} \ = \ \ncr{r+s}{n}. \ee Consider the polynomial \be\label{eq:usefulpolyvandermondeconv} (x+y)^r (x+y)^s\ =\ (x+y)^{r+s}.\ee If we use the binomial theorem to expand the left hand side of \eqref{eq:usefulpolyvandermondeconv}, we get the coefficient of the $x^n y^{r+s-n}$ is the left hand side of \eqref{eq:ncrcombidentityvandermonde}, while if we use the binomial theorem to find the coefficient of $x^n y^{r+s-n}$ on the right hand side of \eqref{eq:usefulpolyvandermondeconv} we get \eqref{eq:ncrcombidentityvandermonde}, which completes the proof.
\end{proof}

\begin{lem}\label{lem:neededlemcomb} Let $\ell, m, s$ be non-negative integers. Then \be \sum_k (-1)^k \ncr{\ell}{m+k} \ncr{s+k}{n} \ = \ (-1)^{\ell+m} \ncr{s-m}{n-\ell}. \ee \end{lem}

\begin{proof} Using $\ncr{a}{b} = \ncr{a}{a-b}$, we rewrite $\ncr{s+k}{n}$ as $\ncr{s+k}{s+k-n}$, and we then rewrite $\ncr{s+k}{s+k-n}$ as $(-1)^{s+k-n} \ncr{-n-1}{s+k-n}$ by using the extension of the binomial coefficient, where we have pulled out all the negative signs in the numerators. The advantage of this simplification is that the summation index is now only in the denominator; further, the power of $-1$ is now independent of $k$. Factoring out the sign, our quantity is equivalent to \bea & & (-1)^{s-n} \sum_k \ncr{\ell}{m+k} \ncr{-n-1}{s+k-n} \nonumber\\ & & \ \ \ = \ (-1)^{s-n} \sum_k \ncr{\ell}{\ell-m-k} \ncr{-n-1}{s+k-n}, \eea where we again use $\ncr{a}{b} = \ncr{a}{a-b}$. By Vandermonde's Convolution, this equals $(-1)^{s-n}$ $\ncr{\ell-n-1}{\ell-m-n+s}$. Using $\ncr{s-m}{\ell-m-n+s}  = \ncr{s-m}{n-\ell}$ and collecting powers of $-1$ completes the proof (note $(-1)^{\ell-m} = (-1)^{\ell+m}$).
\end{proof}

\begin{lem}\label{keycombinatoriallemmaforsuns} Let $m$ be an integer greater than or equal to 1. Then \be \twocase{\sum_{r=0}^m (-1)^r \ncr{m}{r} \ncr{m+r}{r} \frac1{(r+1)(m+r)} \ = \ }{1/2}{if $m=1$}{0}{if $m \ge 2$.} \ee \end{lem}

\begin{proof} The case $m=1$ follows by direct evaluation. Consider now $m \ge 2$.
We have \bea S_m & \ = \ & \sum_{r=0}^m (-1)^r \ncr{m}{r} \ncr{m+r}{r} \frac1{(r+1)(m+r)} \nonumber\\ &= \ & \sum_{r=0}^m (-1)^r \ncr{m}{r} \frac{m+1}{m+1} \ncr{m+r}{r} \frac1{(r+1)(m+r)} \nonumber\\ &\ = \ &
\sum_{r=0}^m (-1)^r \frac{m! (m+1)}{(r+1)\cdot r!m!} \frac1{m+1} \ \frac{(m+r)(m+r-1)!}{r! m \cdot (m-1+r)!}  \frac1{m+r} \nonumber\\ &\ = \ & \sum_{r=0}^m (-1)^r \ncr{m+1}{r+1} \ncr{m-1+r}{r} \frac{1}{m(m+1)}  \nonumber\\ &\ = \ & \frac1{m(m+1)} \sum_{r=0}^m (-1)^r \ncr{m+1}{r+1} \ncr{m-1+r}{m-1}. \eea We change variables and set $u=r+1$; as $r$ runs from $0$ to $m$, $u$ runs from $1$ to $m+1$. To have a complete sum, we want $u$ to start at $0$; thus we add in the $u=0$ term, which is $\ncr{m-2}{m-1}$. As $m \ge 2$, this is 0 from the extension of the binomial coefficient (this is the first of two places where we use $m \ge 2$). Our sum $S_m$ thus equals \bea S_m & \ = \ & -\frac1{m(m+1)} \sum_{u=0}^{m+1} (-1)^u \ncr{m+1}{u} \ncr{m-2+u}{m-1}. \eea We now use Lemma \ref{lem:neededlemcomb} with $k=u$, $m=0$, $\ell = m+1$, $s=m-2$ and $n=m-1$; note the conditions of that lemma require $s$ to be a non-negative integer, which translates to our $m \ge 2$. We thus find \be S_m \ = \ -\frac1{m(m+1)} (-1)^{m+1} \ncr{m-2}{-2} \ = \ 0, \ee which completes the proof.
\end{proof}

We give another proof of Lemma \ref{keycombinatoriallemmaforsuns} below using hypergeometric functions; we thank Frederick Strauch for showing us this approach.

\begin{rek} We present an alternative proof of Lemma \ref{keycombinatoriallemmaforsuns} using the hypergeometric function \be\label{eq:hypergeodef} {}_2F_1(a,b,c;z) \ = \ \frac{\Gamma(c)}{\Gamma(b)\Gamma(c-b)} \int_0^1 \frac{t^{b-1} (1-t)^{c-b-1}dt}{(1-tz)^a}. \ee The following identity for the normalization constant of the Beta function is crucial in the expansions: \be B(x,y) \ = \ \int_0^1 t^{x-1} (1-t)^{y-1}dt \ = \ \frac{\Gamma(x)\Gamma(y)}{\Gamma(x+y)}. \ee We can use the geometric series formula to expand \eqref{eq:hypergeodef} as a power series in $z$ involving Gamma factors. Rewriting $\ncr{m}{r}$ as $(-1)^r \ncr{r-m-1}{r}$, after some algebra we find \be\label{eq:hypergeomprooflemmaa3} S_m\ =\ \frac{\Gamma(m) {}_2F_1(-m,m,2;1)}{\Gamma(2) \Gamma(1+m)} \ = \ \frac{\Gamma(m)}{\Gamma(1+m) \Gamma(2+m) \Gamma(2-m)} \ee (our summation over $r$ in the definition of $S_m$ has become the series expansion of ${}_2F_1(-m,m,2;1)$), where the last step uses \be {}_2F_1(a,b,c;1)\ =\ \frac{\Gamma(c) \Gamma(c-a-b)}{\Gamma(c-a) \Gamma(c-b)} \ee which follows from the normalization constant of the Beta function. Note that the right hand side of \eqref{eq:hypergeomprooflemmaa3} is $1/2$ when $m=1$ and $0$ for $m \ge 2$ because for such $m$, $1/\Gamma(2-m) = 0$ due to the pole of $\Gamma(2-m)$.
\end{rek}

\begin{rek} It is also possible to prove this lemma through symbolic manipulations. Using the results from \cite{PS,PSR}, one may input this into a Mathematica package, which outputs a proof.
\end{rek}


\section{Moments for the family of all curves}\label{sec:momentsallcurves}

Birch \cite{Bi} claims the following: Let \be S_R(p) \ = \ \sum_{a \bmod p}\ \sum_{b \bmod p}\left[ \sum_{x\bmod p} \js{x^3-ax-b}\right]^{2R}. \ee Then for $p \ge 5$,\bea S_1(p) & \ = \ & p^2 \nonumber\\ S_2(p) & = & 2p^3-3p \nonumber\\ S_3(p) & = & 5p^4-9p^2-5p. \eea There are obviously typos here. We know the Legendre sum is at most $2\sqrt{p}$ in absolute value, thus we expect $S_R(p)$ to be on the order of $p^2 \cdot (\sqrt{p})^{2R} = p^{R+2}$; note the powers of $p$ are too low (and they are too high for dividing $S_R(p)$ by the cardinality of the family).

Assuming $S_R(p)$ is a polynomial in $p$, from exploring the results for small $p$ we are led to \bea S_1(p) & \ = \ & p^3 - p^2 \nonumber\\ S_2(p) & = & 2 p^4 - 2 p^3 - 3 p^2 + 3 p \nonumber\\ S_3(p) & = & 5 p^5 - 5 p^4 - 9 p^3 + 4 p^2 + 5 p. \eea Note these are exactly the results from Birch multiplied by $p-1$; we thank Andrew Granville for pointing this out to us. In other words, the formulas in Birch are what remains after dividing by the trivial multiplicative factor $p-1$.

Let $S_R'(p)$ denote the same sum as $S_R(p)$, but with the additional restriction that $4a^3\neq 27b^2$. It is readily seen that $S_R'(p) = S_R(p)+(p-1)$; the reason is that if the discriminant equals zero, then $x^3-ax-b=(x-c)^2(x-d)$ for some $c,d$, and the sum of these Legendre symbols over all $x$ modulo $p$ is $\pm 1$ (the sum is the same as $\sum_{x \not\equiv c \bmod p} \js{x-d} = -\js{c-d} = \pm 1$). Explicitly, we  find \bea S_1(p) & \ = \ & p^3 - p^2 - p + 1 \nonumber\\ S_2(p) & = & 2 p^4 - 2 p^3 - 3 p^2 + 2 p + 1 \nonumber\\ S_3(p) & = & 5 p^5 - 5 p^4 - 9 p^3 + 4 p^2 + 4 p + 1. \eea

As the evaluation of these sums is central to this and other investigations, we provide two proofs of the formula for $S_1(p)$ in the hopes that these arguments will be of use to other researchers studying similar questions.

We first give the proof in \cite{Mil1}. We have the following expansion of $\js{x}$:
\begin{eqnarray}\label{eq:formulafromBEW}
\js{x} \ = \ G_p^{-1} \sum_{c=1}^p \js{c} \ep{cx},
\end{eqnarray}
where $\ep{a} = \exp(2\pi i a/p)$ and $G_p = \sum_{a (p)} \js{a} \ep{a}$, which equals $\sqrt{p}$
for $p \equiv 1 (4)$ and $i\sqrt{p}$ for $p \equiv 3 (4)$. See,
for example, \cite{BEW}.

For the curve $y^2 = f_E(x) = x^3-ax-b$, $a_E(p) = -\sum_{x (p)}
\js{f_E(x)}$. We use \eqref{eq:formulafromBEW} to rewrite $a_E(p)$ as
\begin{eqnarray}\label{eqatpexpansion}
a_E(p)\ =\ -G_p^{-1} \sum_{x (p)} \sum_{c=1}^p \js{c} \ep{c f_E(x)}.
\end{eqnarray}
We take the complex conjugate, which on the RHS introduces a minus
sign into the exponential and sends $G_p$ to $\overline{G_p}$, and
has no effect on the LHS (which is real). The sum becomes
\begin{eqnarray}
S &\ =\ & (G_p \overline{G_p})^{-1} \sum_{a=0}^{p-1}
\sum_{b=0}^{p-1} \prod_{i=1}^{2} \sum_{x_i=0}^{p-1}
\sum_{c_i=0}^{p-1} \js{c_i} \ep{(-1)^{i+1}(c_i x_i^3 - c_i ax_i -
c_i b)} \nonumber\\ & = & \frac{1}{p} \sum_{x_1, c_1=0}^{p-1}
\sum_{x_2, c_2=0}^{p-1} \js{c_1 c_2} \ep{c_1 x_1^3 - c_2 x_2^3}
\sum_{a=0}^{p-1} \ep{-(c_1 x_1 - c_2 x_2)a}  \nonumber\\ & & \ \ \
\cdot \sum_{b=0}^{p-1}\ep{-(c_1 - c_2)b}.
\end{eqnarray}
The $b$-sum vanishes unless $p | (c_1 - c_2)$, which only happens
if $c_1 = c_2 = c$. The $a$-sum vanishes unless $p | (c x_1 - c
x_2)$. As $c \not\equiv 0 (p)$ (we have the factor $\js{c}$) this
forces $x_1 = x_2 = x$. As $c$ is non-zero, $\js{c^2} = 1$, the
first exponential factor is $1$, and the sums collapse to
\begin{eqnarray}
S &\ =\ & \frac{1}{p} \sum_{c=1}^{p-1} 1 \sum_{x=0}^{p-1} 1
\sum_{a=0}^{p-1} 1 \sum_{b=0}^{p-1} 1 \nonumber\\ & = &
\frac{1}{p} (p-1) \cdot p \cdot p \cdot p = p^3 - p^2.
\end{eqnarray}

\begin{rek} We sketch an alternate proof for $S_1(R)$. We have \bea S_1(R) \ = \ \sum_{a\bmod p}\ \sum_{b \bmod p}\ \sum_{x \bmod p}\ \sum_{y\bmod p} \js{x^3-ax-b} \js{y^3-ay-b}.\ \ \ \eea We use the following result: \bea \mathcal{R} & \ = \ & \sum_{n \bmod p} \js{n+c_1} \js{n+c_2} \nonumber\\ & \ = \ & \sum_{n \bmod p} \js{n^2 + n(c_2-c_1)}\nonumber\\ & \ = \ & \sum_{n \bmod p} \js{n^2 + \alpha n (c_2-c_1)} \eea for any $\alpha \not\equiv 0 \bmod p$. Thus \be (p-1)\mathcal{R}\ =\ \sum_{\alpha \not\equiv 0 \bmod p} \sum_{n \bmod p} \js{n^2+\alpha n (c_2-c_1)} \ = \ -(p-1), \ee so $\mathcal{R} = -1$. Thus \be \twocase{\sum_{n\bmod p} \js{n+c_1} \js{n+c_2} \ = \ }{p-1}{if $c_1 \equiv c_2 \bmod p$}{-1}{otherwise.} \ee We rewrite our sum (replacing $a$ with $-a$ and $b$ with $-b$) as \be S_1(R) \ = \ \sum_{a\bmod p}\ \sum_{x \bmod p} \ \sum_{y \bmod p} \left[ \sum_{b \bmod p} \js{b+(x^3+ax)} \js{b+(y^3+ay)}\right]. \ee When is $x^3+ax \equiv y^3 + ay \bmod p$? This is always true if $x = y$ and $a$ is arbitrary, which gives a contribution of $p\cdot p \cdot (p-1)$. If $x\neq y$ (which happens $p^2-p$ times), there is a unique value of $a$ that works, namely $-(x^3-y^3)/(x-y)$. For this special $a$ the contribution is $(p^2-p) \cdot 1 \cdot (p-1)$, and for the other $a$ the contribution is $(p^2-p) \cdot (p-1) \cdot (-1)$. Adding yields $p^3-p^2$.
\end{rek}


\begin{thebibliography}{B-LGHT1}

\bibitem[AL-RM]{AL-RM}
\newblock S. Arms, \'{A}. Lozano-Robledo and S. J. Miller, \emph{Constructing one-parameter families of elliptic curves over $\Q(T)$ with moderate rank}, Journal of Number Theory \textbf{123} (2007), no.
2, 388--402.

\bibitem[BZ]{BZ}
\newblock S. Baier and L. Zhao, \emph{The Sato-Tate conjecture on average for small angles}, Transactions of the AMS \textbf{361} (2009), no. 4, 1811--1832.

\bibitem[BS]{BS}
\newblock W. D. Banks and I. E. Shparlinski, \emph{Sato-Tate, cyclicity, and divisibility statistics on average for elliptic curves of small height}, Israel J. Math. \textbf{173} (2009), 253--277.

\bibitem[B-LGG]{B-LGG}
\newblock T. Barnet-Lamb, T. Gee and D. Geraghty, \emph{The Sato-Tate conjecture for Hilbert modular forms}, preprint (2009). \texttt{http://arxiv.org/abs/0912.1054}.

\bibitem[B-LGHT]{B-LGHT}
\newblock T. Barnet-Lamb, D. Geraghty, M. Harris and R. Taylor, \emph{A family of Calabi-Yau varieties and potential automorphy II.}, preprint (2009). \texttt{http://www.math.harvard.edu/~rtaylor/cy2.pdf}.

\bibitem[BEW]{BEW}
\newblock B. Berndt, R. Evans and K. Williams, \emph{Gauss and
Jacobi Sums}, Canadian Mathematical Society Series of Monographs
and Advanced Texts, vol. \textbf{21}, Wiley-Interscience
Publications, John Wiley \& Sons, Inc., New York, 1998.

\bibitem[Bi]{Bi}
\newblock B. Birch, \emph{How the number of points of an elliptic
curve over a fixed prime field varies}, J. London Math. Soc.
\textbf{43}, $1968$, $57-60$.

\bibitem[BS-D1]{BS-D1}
\newblock B. Birch and H. Swinnerton-Dyer, \emph{Notes on elliptic
curves. I}, J. reine angew. Math. \textbf{212}, $1963$, $7-25$.

\bibitem[BS-D2]{BS-D2}
\newblock B. Birch and H. Swinnerton-Dyer, \emph{Notes on elliptic
curves. II}, J. reine angew. Math. \textbf{218}, $1965$, $79-108$.

\bibitem[BCDT]{BCDT}
\newblock C. Breuil, B. Conrad, F. Diamond and R. Taylor, \emph{On
the modularity of elliptic curves over \textbf{Q}: wild $3$-adic
exercises}, J. Amer. Math. Soc. \textbf{14}, no. $4$, $2001$,
$843-939$.

\bibitem[Bro]{Bro}
\newblock M. L. Brown, Heegner modules and elliptic curves, Lecture Notes In Mathematics, vol. 1849, Springer-Verlag, 2004.

\bibitem[CHT]{CHT}
\newblock L. Clozel, M. Harris and R. Taylor, \emph{Automorphy for some $\ell$-adic lifts of automorphic mod $\ell$ Galois representations}, Publications Math\'{e}matiques de L'IH\'{E}S \textbf{108} (2008), no. 1, 1--181.

\bibitem[CW]{CW}
J. Coates and A. Wiles, \emph{On the conjecture of Birch and Swinnerton-Dyer},  Invent. Math.  \textbf{39}  (1977), no. 3, 223--251.


\bibitem[CDF]{CDF}
\newblock J. Brian Conrey, W. Duke and D. Farmer, \emph{The distribution of the eigenvalues of Hecke operators}, Acta Arithmetica \textbf{78} (1997), no. 4, 405--409.

\bibitem[De]{De}
\newblock P. Deligne, \emph{La conjecture de Weil. II} Inst. Hautes \'Etudes
Sci. Publ. Math. \textbf{52}, $1980$, $137-252$.

\bibitem[Deu]{Deu}
\newblock M. Deuring, \emph{Die Typen der Multiplikatorenringe elliptischer Funktionenk\"{o}rper}, Abh. Math. Sem. Hansischen Univ. \textbf{14} (1941), 197--272.

\bibitem[Di]{Di}
\newblock F. Diamond, \emph{On deformation rings and Hecke rings},
Ann. Math. \textbf{144}, $1996$, $137-166$.

\bibitem[Fe1]{Fe1}
\newblock S. Fermigier, \emph{Z\'eros des fonctions $L$ de courbes
elliptiques}, Exper. Math. \textbf{1}, $1992$, $167-173$.

\bibitem[Fe2]{Fe2}
\newblock S. Fermigier, \emph{\'Etude exp\'erimentale du rang de
familles de courbes elliptiques sur $\Q$}, Exper. Math.
\textbf{5}, $1996$, $119-130$.

\bibitem[GKZ]{GKZ}
\newblock B. H. Gross, W. Kohnen and D. B. Zagier, \emph{Heegner points and derivatives of L-series. II}, Mathematische Annalen \textbf{278} (1987), no. 1–4, 497--562.

\bibitem[GZ]{GZ}
\newblock B. H. Gross and D. B. Zagier, \emph{Heegner points and derivatives of L-series}, Inventiones Mathematicae \textbf{84} (1986), no. 2, 225--320.

\bibitem[HS-BT]{HS-BT}
\newblock M. Harris, N. Shepherd-Barron and R. Taylor, \emph{A family of Calabi-Yau varieties and potential
automorphy}, to appear in the Annals of Math.

\bibitem[He1]{He1}
\newblock M. Hecke, \emph{Eine neue Art von Zetafunktionen und ihre Beziehungen zur Verteilung der Primzahlen, I}, Math. Z. \textbf{1} (1918), 357--376.

\bibitem[He2]{He2}
\newblock M. Hecke, \emph{Eine neue Art von Zetafunktionen und ihre Beziehungen zur Verteilung der Primzahlen, II}, Math. Z. \textbf{6} (1920), 11--51.

\bibitem[Ka]{Ka}
\newblock N. Katz, \emph{Gauss Sums, Kloosterman Sums, and Monodromy Groups}, Princeton University Press, Princeton, NJ 1988.

\bibitem[Kol1]{Kol1}
\newblock V. A. Kolyvagin, \emph{The Mordell-Weil and Shafarevich-Tate groups for Weil elliptic curves}, Izv. Akad. Nauk SSSR Ser. Mat.  \textbf{52}  (1988),  no. 6, 1154--1180, 1327;  translation in  Math. USSR-Izv.  \textbf{33}  (1989),  no. 3, 473--499

\bibitem[Kol2]{Kol2}
\newblock V. A. Kolyvagin, \emph{Finiteness of $E(Q)$ and ${\rm Shah}(E,Q)$ for a subclass of Weil curves}, Izv. Akad. Nauk SSSR Ser. Mat. \textbf{52} (1988), no. 3, 522--540, 670--671; translation in Math. USSR-Izv. \textbf{32} (1989), no. 3, 523--541.

\bibitem[Mic]{Mic}
\newblock P. Michel, \emph{Rang moyen de familles de courbes elliptiques
et lois de Sato-Tate}, Monat. Math. \textbf{120}, $1995$,
$127-136$.

\bibitem[Mil1]{Mil1}
\newblock S. J. Miller, \emph{$1$- and $2$-Level Densities for Families of Elliptic
Curves: Evidence for the Underlying Group Symmetries}, Princeton
University, Ph.\ D. thesis, 2002. \\ \texttt{http://www.williams.edu/go/math/sjmiller/public\underline{\ }html/\\ math/thesis/thesis.html}

\bibitem[Mil2]{Mil2}
\newblock S. J. Miller, \emph{Variation in the number of points on elliptic curves
and applications to excess rank}, C. R. Math. Rep. Acad. Sci. Canada
\textbf{27} (2005), no. 4, 111--120.

\bibitem[Mil3]{Mil3}
\newblock S. J. Miller, \emph{Investigations of zeros near the central point of
elliptic curve $L$-functions}, Experimental Mathematics \textbf{15}
(2006), no. 3, 257--279.

\bibitem[Mil4]{Mil4}
\newblock S. J. Miller, \emph{An identity for sums of polylogarithm functions},
Integers: Electronic Journal Of Combinatorial Number Theory \textbf{8} (2008), \#A15.

\bibitem[MS]{MS}
\newblock M. Ram Murty and K. Sinha, \emph{Effective equidistribution of eigenvalues of Hecke operators}, Journal of
Number Theory \textbf{129} (2009) 681--714.

\bibitem[Na]{Na}
\newblock K. Nagao, \emph{$\Qoft$-rank of elliptic curves and certain
limit coming from the local points}, Manuscr. Math. \textbf{92},
$1997$, $13-32$.

\bibitem[PS]{PS}
\newblock P. Paule and M. Schorn, \emph{A Mathematica Version of Zeilberger's Algorithm
for Proving Binomial Coefficient Identities}, J. Symbolic Computation \textbf{11} (1994), 1-25.

\bibitem[PSR]{PSR}
\newblock P. Paule, M. Schorn and A. Riese, \emph{An Implementation Of Zeilberger's Fast Algorithm}, \texttt{http://www.risc.uni-linz.ac.at/research/\\ combinat/software/PauleSchorn/index.php}

\bibitem[RS]{RS}
\newblock M. Rosen and J. Silverman, \emph{On the rank of an elliptic
surface}, Invent. Math. \textbf{133} (1998), 43--67.

\bibitem[Ru]{Ru}
\newblock K. Rubin, \emph{The one-variable main conjecture for elliptic curves with complex multiplication,
$L$-functions and arithmetic} (Durham, 1989), in London Math. Soc. Lecture Note Series \textbf{153},
Cambridge Univ. Press, Cambridge, 1991, pages 353--371.

\bibitem[Sar]{Sar}
\newblock P. Sarnak, \emph{Statistical properties of eigenvalues of the Hecke operators}, in Analytic
Number Theory and Diophantine Problems (Stillwater, OK, 1984), Progr. Math. 70, Birkh\"{a}user, Boston, 1987, 321--331.

\bibitem[Sel]{Sel}
\newblock A. Selberg, \emph{On the estimation of Fourier Coefficients of Modular Forms}, Proc. Amer. Math. Soc., Symposia in Pure Math. VIII: Theory of Numbers (Pasedena, 1963), 1--15.

\bibitem[Ser]{Ser}
\newblock J.-P. Serre, \emph{R\'{e}partition Asymptotique des Valuers Propres de l'Operateur de Hecke $T_p$}, J. Amer. Math. Soc. \textbf{10} (1997), no. 1, 75--102.

\bibitem[Sh1]{Sh1}
\newblock I. E. Shparlinski, \emph{On the Lang-Trotter and Sato-Tate Conjectures on Average for Some Families of Elliptic Curves}, preprint.

\bibitem[Sh2]{Sh2}
\newblock I. E. Shparlinski, \emph{On the Sato-Tate Conjecture on Average for Some Families of Elliptic Curves}, preprint.

\bibitem[Sil1]{Sil1}
\newblock J. Silverman, \emph{The Arithmetic of
Elliptic Curves}, Graduate Texts in Mathematics \textbf{106},
Springer-Verlag, Berlin - New York, $1986$.

\bibitem[Sil2]{Sil2}
\newblock J. Silverman, \emph{Advanced Topics in the Arithmetic of
Elliptic Curves}, Graduate Texts in Mathematics \textbf{151},
Springer-Verlag, Berlin - New York, $1994$.

\bibitem[ST]{ST}
J. Silverman and J. Tate, \emph{Rational Points on Elliptic
Curves}, Springer-Verlag, New York, 1992.

\bibitem[Ta]{Ta}
\newblock J. T. Tate, \emph{Algebraic cycles and poles of zeta-functions}, in Arithmetical Algebraic geometry (Proc. Purdue Conf. 1963), O. F. G. Schilling (ed.), Harper \& Row, 1965, pp. 93--110.

\bibitem[Tay]{Tay}
\newblock R. Taylor, \emph{Automorphy for some $\ell$-adic lifts of automorphic mod $\ell$ Galois representations. II}, 	 Publications Math\'{e}matiques de L'IH\'{E}S \textbf{108} (2008), no. 1, 183--239.

\bibitem[TW]{TW}
\newblock R. Taylor and A. Wiles, \emph{Ring-theoretic properties of
certain Hecke algebras}, Ann. Math. \textbf{141}, $1995$,
$553-572$.

\bibitem[Ward]{Ward}
\newblock K. J. Ward, \emph{Series Sums of Binomial Coefficients}, webpage: \texttt{http://www.trans4mind.com/personal\underline{\ }development/\\ mathematics/series/summingBinomialCoefficients.htm}

\bibitem[Wa]{Wa}
\newblock L. Washington, \emph{Class numbers of the simplest
cubic fields}, Math. Comp. \textbf{48}, number $177$, $1987$,
$371-384$.

\bibitem[Wi]{Wi}
Wikipedia, \emph{Chebyshev Polynomials}, June 21, 2009. \\ \texttt{http://en.wikipedia.org/wiki/Chebyshev\underline{\ }polynomials}.

\ \\

\end{thebibliography}
\end{document}